\documentclass[notitlepage,twoside,a4paper]{amsart}
\usepackage{mathrsfs}
\usepackage{amsmath,amssymb,enumerate}
\usepackage{epsfig,fancyhdr,color}
\usepackage{epstopdf}
\usepackage{amssymb}
\usepackage{amsmath,amsthm}
\usepackage{latexsym}
\usepackage{amscd}
\usepackage{psfrag}
\usepackage{graphicx}
\usepackage{epsf}
\usepackage[latin1]{inputenc}
\usepackage[all]{xy}
\usepackage{tikz}
\usepackage{prettyref}
\usepackage{subfigure}
\usepackage{float}
\usepackage{color}
\usepackage{array}
\usepackage{cite}
\usepackage[totalwidth=16cm,totalheight=24cm]{geometry}
\newcommand{\II}{I\hspace{-0.1cm}I}
\newcommand{\III}{I\hspace{-0.1cm}I\hspace{-0.1cm}I}

\newtheorem{theorem}{\rm\bf Theorem}[section]
\newtheorem{proposition}[theorem]{\rm\bf Proposition}

\newtheorem{lemma}[theorem]{\rm\bf Lemma}
\newtheorem{corollary}[theorem]{\rm\bf Corollary}
\newtheorem{definition}[theorem]{\rm\bf Definition}

\newtheorem{question}[theorem]{\rm\bf Question}

\newtheoremstyle{named}{}{}{\itshape}{}{\bfseries}{.}{.5em}{#1 \thmnote{#3}}
\theoremstyle{named}

\newcommand{\C}{{\mathbb C}}
\newcommand{\CP}{{\mathbb CP}}
\newcommand{\N}{{\mathbb N}}
\newcommand{\DD}{{\mathbb D}}
\newcommand{\HH}{{\mathbb H}}
\newcommand{\R}{{\mathbb R}}

\newcommand{\cD}{{\mathcal D}}

\newcommand{\Isom}{\rm{Isom}}

\newcommand{\tr}{\mbox{tr}}

\newcounter{notes}%

\def\interieur#1{\mathord{\mathop{\kern 0pt #1}\limits^\circ}}

\title[The Weyl problem for unbounded domain]{The Weyl problem for unbounded convex domains in $\HH^3$}

\author{Jean-Marc Schlenker}
\address{Jean-Marc Schlenker:
University of Luxembourg, FSTM, Department of Mathematics, 
Maison du nombre, 6 avenue de la Fonte,
L-4364 Esch-sur-Alzette, Luxembourg}
\email{jean-marc.schlenker@uni.lu}
\thanks{Partially supported by FNR project O20/14766753.}

\date{v2, \today}

\begin{document}

\begin{abstract}
  Let $K\subset \HH^3$ be a convex subset in $\HH^3$ with smooth, strictly convex boundary. The induced metric on $\partial K$ then has curvature $K>-1$. It was proved by Alexandrov that if $K$ is bounded, then it is uniquely determined by the induced metric on the boundary, and any smooth metric with curvature $K>-1$ can be obtained.

  We propose here an extension of the existence part of this result to unbounded convex domains in $\HH^3$. The induced metric on $\partial K$ is then clearly not sufficient to determine $K$. However one can consider a richer data on the boundary including the ideal boundary of $K$. Specifically, we consider the data composed of the conformal structure on the full boundary of $K$ (in the Poincar\'e model of $\HH^3$), together with the induced metric on $\partial K$. We show that a wide range of ``reasonable'' data of this type, satisfying mild curvature conditions, can be realized on the boundary of a convex subset in $\HH^3$.
  
  %% We do not consider here the uniqueness of a convex subset with given boundary data.

\end{abstract}

	\maketitle

        \tableofcontents
        
    \section{Introduction and main results}

\subsection{The classical Weyl problem in Euclidean and hyperbolic space}
\label{ssc:weyl}

The study of smooth or polyhedral convex surfaces in the 3-dimensional Euclidean space $\R^3$ has a long history. One of the first ``modern'' results on this theme was the proof by Legendre \cite{legendre} and Cauchy \cite{cauchy} of the rigidity of convex polyhedra: if two polyhedra have the same combinatorics and corresponding faces are isometric, then they are congruent.

Alexandrov \cite{alex} later realized that Cauchy's rigidity result is related to a much deeper statement. The induced metric on the boundary of a convex polyhedron in $\R^3$ is an Euclidean metric with cone singularities with cone angle less than $2\pi$ at the vertices, and Alexandrov proved that each metric of this type on the sphere can be realized as the induced metric on the boundary of a unique polyhedron. This result was later extended by Alexandrov \cite{alex} to bounded convex subset of $\R^3$, thus solving the ``Weyl problem'', proposed by H. Weyl in 1915 and for which substantial progress had already been made by H. Lewy \cite{lewy1935priori} and others.

\begin{theorem}[Lewy, Alexandrov] \label{tm:alexE}
  Any smooth metric of positive curvature on the sphere is induced on the boundary of a unique smooth strictly convex subset of $\R^3$. 
\end{theorem}

% Here (and elsewhere) by ``strictly convex'' we mean that the shape operator of the boundary is positive definite. By ``smooth'' we mean $C^\infty$, although different degrees of regularity could be considered.

% This result was then further extended by Pogorelov \cite{Po} to non-smooth metrics: the induced metric on the boundary of any bounded convex subset of $\R^3$, even if non-smooth, is of curvature $K\geq 0$ in the sense of Alexandrov, and each such metric on the sphere is realized on the boundary of a unique bounded convex body.

%\subsection{The Weyl problem and its dual for closed convex surfaces in hyperbolic space}
%\label{ssc:weyl-h}

Alexandrov and Pogorelov also extended Theorem \ref{tm:alexE} to hyperbolic space.

\begin{theorem}[Alexandrov, Pogorelov] \label{tm:alexH}
  Any smooth metric of curvature $K> -1$ on the sphere is induced on the boundary of a unique convex subset of $\HH^3$ with smooth boundary. 
\end{theorem}

Versions of Theorem \ref{tm:alexH} also hold for ideal polyhedra \cite{rivin-comp} and for hyperideal polyhedra \cite{shu} (see also \cite{rsc} for a result on convex domains bounded by disks, as well as a broad generalization by Luo and Wu \cite{luo-wu}). The existence of those results hints at the possibility of a wider extension to more general unbounded convex domains in $\HH^3$ as considered here.

\subsection{The boundary data on unbounded convex subsets}

Before stating our main result, we define a notion of ``boundary data'' that is adapted to complete convex subsets of $\HH^3$. In the special case of bounded convex subsets, it coincides with the induced metric in the usual sense.

We consider a convex subsets $\Omega\subset \HH^3$, which a priori is unbounded. Let $\partial_0\Omega=\partial\rho(\Omega)$, where $\rho:\HH^3\to \R^3$ is the Poincar\'e model of $\HH^3$. So $\partial_0\Omega$ is the ``full'' boundary of $\Omega$, including its metric boundary in $\HH^3$ but also its ideal boundary.  

We now define a notion of boundary data for a convex subset of $\HH^3$. When this subset is bounded, it's just the induced metric on the boundary. For an unbounded convex domain $\Omega$, however, a boundary data records both the induced metric on the metric boundary of $\Omega$ and the conformal structure on the ``full'' boundary $\partial_0\Omega$, coming from the induced metric on the metric boundary but also from the conformal structure at infinity on the ideal boundary.

\begin{definition} \label{df:boundary}
  Let $\Omega\subset \HH^3$ be a convex subset. A {\em boundary data} for $\Omega$ is a pair $(U,h)$, where $U\subset \CP^1$ is an open subset, $h$ is a complete, conformal metric on $U$, such that there exists a homeomorphism $\gamma:\partial_0\Omega\to \CP^1$ such that:
  \begin{itemize}
  \item $\gamma$ is conformal on the interior of $\partial_\infty \Omega$,
  \item $\gamma\circ\rho(\partial\Omega)=U$, and $\gamma\circ \rho$ is an isometry from $\partial \Omega$ equipped with its induced metric to $(U,h)$. 
  \end{itemize}
\end{definition}

The hyperbolic Gauss formula shows that $h$ has curvature $K\geq -1$. Note that $\partial_0\Omega=\partial\Omega\cup \partial_\infty\Omega$, and a homeomorphism $\gamma$ as above defines a conformal structure on $\partial\Omega$ and on $\partial_\infty\Omega$, by pull-back by $\gamma$ of the conformal structure on $\CP^1$.

We do not consider here whether a boundary data (and homeomorphism $\gamma$) is always unique, see \cite[\S 1.4]{weylsurvey} for partial results in this direction. We note however that one always exists under conditions that cover the hypothesis of Theorem \ref{tm:main} below. The statement uses the following definition.

\begin{definition} \label{df:bounded}
  Let $h$ be a complete, smooth, conformal metric on a surface $S$ of curvature $K$. We will say that $h$ has {\em bounded derivatives} if for each $k\geq 1$, there exists a constant $C_k>0$ such that the $k$-th derivatives of $K$, measured with respect to $h$, are bounded by $C_k$.
  %%\marginnote{Replace by previous hyp on conformal factor, in each quasidisk. Removed footnote.}
  %% \footnote{This definition is somewhat different from, and weaker than, the definition used in \cite{weylsurvey}.}
\end{definition}

\begin{lemma} \label{lm:existence}
  Let $\Omega\subset \HH^3$ be a convex subset such that $\partial\Omega$ is a disjoint union of topological disks, and that the induced metric on $\partial \Omega$ has 
  \begin{enumerate}
  \item curvature $K\in [-1+\epsilon,1/\epsilon]$, for some $\epsilon>0$,
  \item bounded derivatives,
  \item injectivity radius bounded from below by $\epsilon$. 
  \end{enumerate}
  Then there exists a boundary data for $\Omega$. 
\end{lemma}

\begin{proof}
  Let $\delta:\partial_\infty\HH^3\to \partial_0\Omega$ be the closest-point projection -- in other terms, $\delta$ is the identity on $\partial_\infty\Omega$, while it projects each point of $\partial_\infty\HH^3\setminus \partial_\infty\Omega$ to a point of $\partial\Omega$.

  It is proved in Lemma \ref{lm:p-curvature} that, under the hypothesis considered here, $\partial\Omega$ has principal curvatures in $[1/k_0, k_0]$, for some $k_0>1$.

  As a consequence, $\delta$ is $K_0$-quasi-conformal, for a constant $K_0$ depending only on $k_0$. Applying the measurable Riemann mapping theorem \cite{ahlfors-bers:riemann} to the Beltrami differential of $\delta$ on $\partial_\infty\HH^3\setminus \partial_\infty\Omega$, we obtain a quasi-conformal homeomorphism $q:\partial_\infty\HH^3\to \CP^1$ which is conformal on $\partial_\infty\Omega$, and has the same Beltrami differential as $\delta$ on $\partial_\infty\HH^3\setminus \partial_\infty\Omega$.

  The homeomorphism $\delta\circ q^{-1}:\CP^1\to \partial_0\Omega$ is then conformal by construction. The pair $((\delta\circ q^{-1})^{-1}(\partial\Omega), (\delta\circ q^{-1})^*I)$ is therefore a boundary data for $\Omega$, where $I$ denotes the induced metric on $\partial\Omega$.
\end{proof}

Our main motivation is the following question, appearing as Question $W^{\Omega}_{\HH^3}$ in \cite{weylsurvey}, reformulated here in line with Definition \ref{df:boundary}.

\begin{question} \label{q:weyl}
  Let $U\subset \CP^1$ be an open subset, and let $h$ be a complete metric of curvature $K\geq -1$ on $U$. Is there a unique convex subset $\Omega\subset \HH^3$ such that $(U,h)$ is a boundary data for $\Omega$? 
\end{question}

As pointed out in \cite{weylsurvey}, the answer to this question in such a general setting might be negative, and additional assumptions might be necessary. We provide here an existence result under stronger curvature conditions, and under a mild regularity asumption on the boundary of $U$ together with a topological hypothesis on $U$. 

\begin{theorem} \label{tm:main}
  Let $k_0>0$, $U\subset \CP^1$ be an open subset which is a disjoint union of $k_0$-quasidisks, let $\epsilon>0$, and let $h$ be a smooth, complete metric or curvature $K\in [-1+\epsilon,-\epsilon]$ on $U$ with bounded derivatives. %% \marginnote{added hypothesis. Check: need quasidisks?}.
  There exists a convex subset $\Omega\subset \HH^3$ such that the boundary data of $\Omega$ is $(U,h)$. 
\end{theorem}

A {\em quasidisk} in $\CP^1$ is the image of a disk by a quasiconformal map from $\CP^1$ to $\CP^1$.

\subsection{A dual question}

Specifically in $\HH^3$, there is also a {\em dual} Weyl problem, concerning the third fundamental form on the boundary of convex bodies, rather than their induced metrics. This third fundamental form is in fact the induced metric on the boundary of a dual convex body in the de Sitter space. 

For ideal polyhedra, the third fundamental form is encoded in the (exterior) dihedral angles, so that the dual Weyl problem corresponds to determining which families of dihedral angles can be obtained, see \cite{Andreev-ideal,rivin-annals}. For bounded polyhedron, it corresponds to the ``dual metric'' considered by Hodgson and Rivin \cite{HR}, while a result is also known for smooth metrics \cite{these}.

Question \ref{q:weyl} has an analog concerning the extension of this ``dual'' Weyl problem to unbounded convex domains with a prescribed data at infinity that includes the third fundamental form, rather than the induced metric, on $\partial\Omega$, see \cite{weylsurvey}. We do not consider this dual problem here. It appears plausible that the methods developed here to prove Theorem \ref{tm:main} could also apply, in some form, to prove a dual statement, but a different regularization procedure appears to be needed to obtain the necessary estimates.

\subsection{Related results and motivations}

A list of related results as well as some background information can be found \cite{weylsurvey}, we only list here a few results and conjectures in direct relation with Theorem \ref{tm:main}.

\subsubsection*{Convex subsets with a quasicircle at infinity}

One of the simplest cases that can be considered is when $U$ is the complement of a quasi-circle in $\CP^1$, and $h$ is a complete metric of constant curvature $K\in (-1,0)$ on $U$. The corresponding special cases of Theorem \ref{tm:main} (together with a dual result with the induced metric replaced by the third fundamental form) is the main result of \cite{convexhull}. (Note that the result of \cite{convexhull} on the induced metrics also covers the case $K=-1$, where the convex subset $\Omega$ is the convex hull of its ideal boundary.) Recently, this existence result was extended to metrics of variable curvature by Abderrahim Mesbah \cite{mesbah2024}.

The statements of Theorem A and Theorem B in \cite{convexhull} appear superficially quite different from the statement of Theorem \ref{tm:main}, since it concerns the possibility to prescribe the ``gluing map'' at infinity between the induced metrics on the two connected components of the boundary of a domain, with induced metric of constant curvature, which has as boundary at infinity a quasicircle. A moment of reflection shows however that Theorem B of \cite{convexhull} is the special case of Theorem \ref{tm:main} when $U$ is the complement of a quasicircle in $\CP^1$ and $h$ has constant curvature $K\in (-1,0)$, with the choice of a gluing map equivalent, through conformal welding, to the choice of the quasicircle which is the complement of $U$ in $\CP^1$ .  

\subsubsection*{Convex subsets in quasifuchsian manifolds}

The main motivation for this result from \cite{convexhull} is a conjecture made by Thurston concerning the induced metric and the measured bending lamination on the convex cores of geometrically finite hyperbolic manifolds. We recall those conjectures in the quasifuchsian case for simplicity. A {\em quasifuchsian manifold} is a complete hyperbolic manifold, homeomorphic to $S\times \R$, where $S$ is a closed surface of genus at least $2$, containing a compact, non-empty, geodesically convex subset. A quasifuchsian manifold $M$ contains a smallest non-empty convex subset, its {\em convex core} $C(M)$, which is either a totally geodesic surface (in the ``Fuchsian'' case) or a closed domain with boundary the disjoint union of two surfaces, each homeomorphic to $S$, with an induced metric of constant curvature $-1$, pleated along a measured laminations. Thurston thus stated two conjectures:
\begin{enumerate}
\item any pair of hyperbolic metrics on $S$, considered up to isotopy, can be uniquely realized as the induced metrics on the boundary of a quasifuchsian manifold,
\item any pair of measured laminations on $S$ that fill together, such that no leaf has weight larger than $\pi$, can be realized uniquely as the measured bending lamination on the convex core of a quasifuchsian manifold.
\end{enumerate}
The existence part of both statements is known, for (1) it is a consequence of results of Labourie \cite{L4} or of Epstein--Marden \cite{epstein-marden}, while for (2) it is proved (for geometrically finite hyperbolic manifolds) by Bonahon--Otal \cite{bonahon-otal}. Uniqueness was recently obtained for (2) \cite{uniqueness}, while it remains open for (1).

Existence and uniqueness results are known for larger geodesically convex subsets (containing the convex core) with smooth boundary. The induced metric then has curvature $K>-1$, while the measured bending lamination is replaced by the third fundamental form of the boundary, which has curvature $K<1$ and closed, contractible geodesics of length $L>2\pi$. It is proved in \cite{L4} (for the existence) and \cite{hmcb} (for the uniqueness) that, if $h_-, h_+$ are two smooth metrics of curvature $K>-1$ on a closed surface $S$ of genus at least $2$, there exists a unique hyperbolic metric with convex boundary on $S\times [0,1]$ such that the induced metrics on the two connected components of the boundary are $h_-$ and $h_+$.

\subsubsection*{Convex domains in convex co-compact hyperbolic manifolds}

The two Thurston conjectures above can in fact be stated for convex co-compact hyperbolic manifolds, and in this more general setting the existence has been proved for both conjecture, and are due to the authors mentioned above in the quasifuchsian setting.

Moreover, if $M$ is a closed manifold with non-empty boundary which admits a convex co-compact hyperbolic manifold on its interior, then, again by \cite{L4,hmcb}, any metric of curvature $K>-1$ on the boundary of $M$ can be uniquely realized as the induced metric of a convex hyperbolic metric $g$ on $M$.

Theorem \ref{tm:main} can be seen as a (partial) generalization of the existence part of this result. Indeed, given such a hyperbolic metric $g$ on $M$, the universal cover $\tilde{M}$ of $M$, equipped with the lifted metric $\tilde g$, is isometric to the interior of a convex domain $\Omega\subset \HH^3$, with $\partial_\infty\Omega$ equal to the limit set $\Lambda$ of $M$. $\Omega$ is then invariant under the action of a representation $\rho:\pi_1M\to PSL(2,\C)$. The existence result from \cite{L4} then implies the existence of a convex domain $\Omega\subset \HH^3$ with a prescribed metric of curvature $K>-1$ on each boundary component (equal to the lift of a metric on the corresponding boundary component of $\partial M$). Moreover, in $\partial_0\Omega$ equipped with the underlying conformal structure, each connected component of $\partial\Omega$ corresponds to a quasidisk. So this result from \cite{L4} implies a special case of Theorem \ref{tm:main}, which moreover is invariant under the representation $\rho$ (without the asumption that the curvature is negative). Theorem \ref{tm:main} can be considered as extending the existence result of \cite{L4} without 
the need for a group action. However Theorem \ref{tm:main} is more general since it allows for situations where $U$ is not dense in $\CP^1$. 

\subsubsection*{Convex domains complete on one side in quasifuchsian manifolds}

One such situation can be found by considering a closed, geodesically convex, subset $\Omega$ in a quasifuchsian manifold $M$, assuming that $\Omega$ is complete ``on one side'', that is, $\Omega$ is homeomorphic to $S\times  [0,\infty)$, where $S$ is again a closed surface of genus at least $2$ and $\Omega$ is complete on the side corresponding to $\infty$, while its metric boundary corresponds to $S\times \{ 0\}$, which is a smooth, locally strictly convex surface $\partial\Omega$  in $M$. The induced metric on $\partial\Omega$ then has curvature $K>-1$.

One can then ask whether given any conformal structure $c$ on $S$, and any metric $h$ of curvature $K>-1$ on $S$ (both considered up to isotopy) there is a unique quasifuchsian metric on $S\times \R$, containing a unique geodesically convex subset $\Omega$ complete on one side, such that the conformal structure on the complete boundary is $c$, while the induced metric on the other boundary is $h$. This case is considered in \cite{chen2022geometric}, where an existence and uniqueness result is obtained.

Lifting to the universal cover of $\Omega$, this translates as another instance of Theorem \ref{tm:main}, where $U$ is now a disk in $\CP^1$ (corresponding to the universal cover of $\partial\Omega$) while the interior of the complement of $U$ corresponds to the universal cover of $\partial_\infty\Omega$. An existence result in this setting is proved in \cite{chen2022geometric}, by methods which are quite different from those considered here.

%\subsubsection*{The Koebe circle domain conjecture}

% === ref to Luo and al. 

\subsection{Questions not considered here}

The Weyl problem described in Section \ref{ssc:weyl} can also be stated in terms of isometric immersions of locally convex surfaces in $\R^3$ or $\HH^3$. This point of view leads to different extensions to unbounded locally convex surfaces in $\HH^3$, see \cite[Section 5]{weylsurvey}. We do not consider questions arising from this other point of view here.

It should be noted that the curvature hypothesis in Theorem \ref{tm:main} are quite restrictive. On one hand, one could consider metrics of curvature $K\geq -1$, including the special case of $K=-1$, corresponding to the case where $\partial\Omega$ is a locally convex pleated surface, and $\Omega$ is the convex hull of its asymptotic boundary $\partial_\infty\Omega$. On the other hand, one could consider metrics with positive curvature at some points -- this is excluded here for ``technical'' reasons. The hypothesis of bounded derivatives is also probably too restrictive, and could be weakened with additional efforts.

We also do not consider here the uniqueness in Theorem \ref{tm:main}, since different methods and tools are necessary.

\subsection{Outline of the proofs and content of the paper}

Section \ref{sc:background} contains a number of well-known constructions and results that are needed at various places in the paper. Section \ref{sc:approximation} shows how a given boundary data $(U,h)$ can be ``approximated'', in a suitable sense, by a sequence $(h_n)_{n\in \N}$ of Riemannian metrics on $S^2$ which are equal to $h$ on increasing domains in $U$, and have constant positive curvature (going to $0$ as $n\to \infty$) on open subsets containing the complement of $U$. The Alexandrov theorem can be applied to those metrics, yielding a sequence $(\Omega_n)_{n\in \N}$ of convex subsets of $\HH^3$, and the remaining part of the paper is focused on proving that $\Omega_n\to \Omega$, where $\Omega$ satisfies the conditions of Theorem \ref{tm:main}.

Section \ref{sc:bounds} provides a key estimate on the principal curvatures of surfaces under bounds on the curvature of the induced metrics, following well-understoods arguments that need to be adapted to our context.

Section \ref{sc:proofs} then contains the proof of the main results.  One key idea of the proof of Theorem \ref{tm:main} is that, thanks to the bounds on the principal curvatures in Section \ref{sc:bounds}, the Gauss maps of the surfaces $\partial\Omega_n$ can be considered as uniformly quasi-conformal diffeomorphisms from $\CP^1$ to $\CP^1$. This sequence therefore has a subsequence converging to a quasiconformal homeomorphism, while, for each connected component of $U$, the corresponding sequence of surfaces in $\HH^3$ (suitably normalized by isometries) has a subsequence converging on compact subsets of $\HH^3$ to a limit surface. The relative positions of those limit surfaces is controlled by the Gauss maps of the $\partial\Omega_n$, leading to the proof of Theorem \ref{tm:main}.

\subsection*{Acknowledgements}

The author is very grateful to Francesco Bonsante, Qiyu Chen, Jeff Danciger, Sara Maloni and Andrea Seppi for many discussions related more or less closely to the result presented here. He would also like to thank Abderrahim Mesbah for pointing out an issue in a previous version of the text.

%%% similar description for dual result %% Introduction and results

    \section{Background}
\label{sc:background}

This section contains two elementary results, used below, recalled here for the reader's convenience and to avoid ambiguities in notations.

\subsection{Conformal metrics on surfaces}
\label{ssc:conformal}

Consider a surface $S$ equipped with a smooth Riemannian metric $h$. The following lemma indicates how the curvature changes under a conformal change of metrics.

\begin{lemma} \label{lm:conformal}
  Let $(S,h)$ be a surface equipped with a Riemannian metric, let $u:S\to \R$ be a smooth function, and let $\bar h=e^{2u}h$. The curvature of $\bar h$ is equal to
  $$ \bar K = e^{-2u}(\Delta_hu+K)~, $$
  where $K$ is the curvature of $h$.
\end{lemma}

Note that in this formula and elsewhere, the Laplacian of $h$ is defined as
$$ \Delta_h u = -\tr(Ddu)~, $$
where $D$ is the Levi-Civita connection of $h$. In other terms, $\Delta_h$ is a positive operator on closed surfaces.

We refer to \cite[Theorem 1.159]{Be} for a proof.

\subsection{Surfaces in $\HH^3$}
\label{ssc:III}

Let $S\subset \HH^3$ be a smooth, oriented surface. We define its {\em shape operator} $B$ as
$$ Bx = \nabla_x n~, $$
where $\nabla$ denotes the Levi-Civita connection of $\HH^3$, $n$ is the unit oriented normal to $S$, and $x$ is any tangent vector to $S$. We then define the {\em second fundamental form} of $S$ as
$$ \II(x,y)=I(Bx,y)~, $$
where $I$ denotes the induced metric on $S$ (also classically known as the first fundamental form) and $x,y$ are two tangent vectors at the same point, and the {\em third fundamental form} of $S$ as
$$ \III(x,y)=I(Bx,By)~. $$
The second fundamental form $\II$ is a symmetric bilinear form on $TS$, and it is positive semi-definite when $S$ is locally convex, with unit normal oriented towards the concave side. We will say that $S$ is {\em strictly convex} if $\II$ is positive definite.

The shape operator $B$ of a smooth surface $S$ in $\HH^3$ satisfies the {\em Gauss-Codazzi} equation:
$$ \det(B)=K-1~, $$
where $K$ is the curvature of the induced metric $I$, and, for any two vector fields $x,y$ tangent to $S$,
$$ (\nabla_x B)(y)-(\nabla_yB)(x)=0~. $$
% \subsection{The de Sitter space $\dS^3$}
% \subsection{The duality between surfaces in $\HH^3$ and $\dS^3$}
% \label{ssc:duality}
% \subsection{The length of closed curves on convex surfaces in $\dS^3$}
% \label{ssc:length}

 %% Background and technical precisions

    \section{Approximation by metrics on the sphere}
\label{sc:approximation}

In this section we define a sequence of conformal metrics on $\CP^1$, that ``approximate'' in a suitable sense a metric $h$ defined on an open subset $U$, as seen in the statement of Theorem \ref{tm:main}. We then show that this sequence of approximating metrics has suitable bounds on curvatures. This will allow us in Section \ref{sc:proofs} to prove that one can apply the Alexandrov Theorem to those metrics and obtain a sequence of convex domains converging (after extraction of a subsequence) to a non-compact domain having the required boundary data.

\subsection{Bounds on conformal factors}

The main goal of this section is to obtain a bound on the gradient of the conformal factor on each connected component of $U$. We consider a fixed spherical metric $c_1$ on $\CP^1$, and a connected component $U_0$ of $U\subset \CP^1$, equipped with a complete Riemannian metric $h$ which can be written as
$$ h =e^{2u}c_1~, $$
where $u:U_0\to \R$. Since $c_1$ is a spherical metric on $\partial_\infty \HH^3$, it is the visual metric associated to a point $m_1\in \HH^3$. Recall that given a point $m\in \HH^3$, the visual metric $c_m$ on $\partial_\infty \HH^3$ associated to $m$ is the spherical metric such that the distance between to ideal points $\xi, \eta\in \partial_\infty \HH^3$ is the angle between the geodesic rays from $m$ to $\xi$ and $\eta$. To simplify a bit the argument, we will assume that $m_1\in CH(\CP^1\setminus U_0)$, where $CH$ denotes the hyperbolic convex hull.

\begin{lemma} \label{lm:conformal-bounds}
  Assume that $h$ has curvature $K\in [-1+\epsilon,-\epsilon]$, where $\epsilon>0$. Assume moreover that  $U_0$ is a $k$-quasi-disk. Then there exists a constant $C>0$, depending only on $k$,  such that for all $m\in U_0$, $\| du(m)\|_h\leq C$.
\end{lemma}

The proof is based on comparisons between four conformal metrics on $U_0$. In addition to $h$ and $c_1$, we will consider the hyperbolic metric $h_{-1}$ (conformal to $h$) on $U_0$, and the Thurston metric $h_{Th}$ on $U_0$.

Recall that the Thurston metric is defined as follows: for all $m\in U_0$, we let $D$ be a maximal disk containing $m$ and contained in $U_0$, and let $h_{Th}$ at $m$ be equal to the complete hyperbolic metric on $D$ at $m$. This metric $h_{Th}$ then has curvature in $[-1,0]$, see e.g. \cite{kulkarni-pinkall}.

We will denote by $v,w,x:U_0\to \R$ the functions such that
$$ h=e^{2v}h_{-1}~,~~ h_{-1}=e^{2w}h_{Th}~, ~~ h_{Th}=e^{2x}c_1~. $$
The proof of Lemma \ref{lm:conformal-bounds} will directly follows from the following three lemmas.

\begin{lemma} \label{lm:v}
  Let $h_{-1}$ be the complete hyperbolic metric in the conformal class of $h$ on $U_0$. Suppose that $h=e^{2v}h_{-1}$, and that $h$ is complete and has curvature in $[-1+\epsilon,-\epsilon]$. There exists a constant $C_v>0$, depending only on $\epsilon$, such that, on $U_0$, $v\in [0,C_v]$ and $\| dv\|_{h_{-1}}\leq C_v$. 
\end{lemma}

\begin{lemma} \label{lm:w}
  Suppose that $U_0$ is a $k$-quasidisk. Then there exists a constant $C_w>0$, depending only on $k$, such that, on $U_0$, $w\in [-C_w,0]$ and $\| dw\|_{h_{-1}}\leq C_w$. 
\end{lemma}

\begin{lemma} \label{lm:x}
  There exists a constant $C_x>0$ such that for all $m\in U_0$, $\| dx\|_{h_{Th}}\leq C_x$.
\end{lemma}

\begin{proof}[Proof of Lemma \ref{lm:conformal-bounds}]
  The lemma clearly follows from Lemma \ref{lm:v}, Lemma \ref{lm:w} and Lemma \ref{lm:x}, since by construction $u=v+w+x$ and all three functions have bounded derivative (measured with respect to $h$) and $v$ and $w$ are bounded, so that $h$, $h_{-1}$ and $h_{Th}$ are within bounded multiplicative constant of each other.
\end{proof}

\begin{proof}[Proof of Lemma \ref{lm:v}]
  The upper bound on $v$ follows from the main result in \cite{ahlfors:extension}. The lower bound on $v$ is a consequence of the 1-dimensional case of the main result of \cite{yau:schwarz} (see also \cite[Theorem 1]{royden:ahlfors} or \cite[Theorem 6]{mateljevic2003ahlfors}).

  To obtain the bound on $dv$, we recall the equation satisfied by $v$ coming from the fact that $h_{-1}=e^{-2v}h$, see e.g. \cite[\S 1.159]{Be}:
  $$ \Delta_h v=K+e^{-2v}~, $$
  where $K$ is the curvature of $h$. Since both $K$ and $v$ are bounded, $\Delta_h v$ is bounded, and it then follows from standard arguments and from the bounds on $v$ and $\Delta_h v$ that $\| dv\|_h$ is bounded. Since $v$ is bounded, $\| dv\|_{h_{-1}}$ is also bounded.

  (We outline this estimate of $\| dv\|_h$ for completeness. Let $x\in U_0$, let $r>0$, and let $B(x,r)$ be the ball of center $x$ and radius $r$ in $(U_0, h)$. We can decompose the restriction of $v$ on $B(x,r)$ as
  $$ v=v_0+v_1~, $$
  where $v_0$ is the harmonic function on $B(x,r)$ with the same boundary function as $v$ on $\partial B(x,r)$, while $v_1$ is the solution of
  $$ \Delta_h v_1=\Delta_h v $$
  on $B(x,r)$ with Dirichlet boundary conditions. Then $\| dv_0(x)\|_h$ is bounded by a fixed constant by the gradient estimate for harmonic functions (because the boundary values of $v_0$ are bounded), while $\| dv_1(x)\|_h$ is bounded by a fixed constant, because $|\Delta_h v|$ is bounded on $B(x,r)$ and $v_1$ can be expressed from $\Delta_hv$ in terms of Green functions of $\Delta_h$ on $B(x,r)$.)
\end{proof}

\begin{proof}[Proof of Lemma \ref{lm:w}]
  We denote by $\DD$ the unit disk in $\C$. We first show that for all $k>1$, there exists a constant $c_w$ such that if $\cD$ is a $k$-quasidisk containing $\DD$ as a maximal disk, if $h_{-1}$ and $h_{Th}$ are the hyperbolic metric and Thurston metrics on $\cD$, respectively, with $h_{-1}=e^{2\bar w}h_{Th}$, then $|\bar w|\leq c_w$ and $\| d\bar w\|_h\leq c_w$ at $0$. This follows from a simple compactness argument. Assume on the contrary that there exists a sequence $\cD_n$ of $k$-quasicircles such that $|\bar w_n|\geq n$ or $\|d\bar w_n\|_{h_n} \geq n$ (where the notations with index $n$ are attached to the quasidisk $\cD_n$). Then by the compactness property of $k$-quasidisks (see e.g. \cite{ahlfors}) there exists a subsequence of $(\cD_n)_{n\in \N}$ which converges in the Hausdorff topology to a $k$-quasidisk $\cD_\infty$ containing $\DD$ as a maximal disk. Then $|\bar w_\infty|$ or $\| d\bar w_\infty\|_{h_\infty}$ would need to be infinite, a contradiction.

  Now let $m\in U_0$, and let $D_m$ be a maximal disk in $U_0$ containing $m$. There exists an element of $PSL(2,\C)$ sending $D_m$ to $\DD$ and $m$ to $0$. Both $h_{-1}$ and $h_{Th}$ are preserved under M\"obius transformation, and the lemma therefore follows from the estimate obtained at $0$ for quasidisks containing $\DD$ as a maximal disk.
\end{proof}

The proof of Lemma \ref{lm:x} requires a simple proposition concerning the comparison between visual metrics at two different points in $\HH^3$. 

\begin{proposition} \label{pr:m0m1}
  Let $m_0, m_1\in \HH^3$ be at distance $\delta>0$, and let $y:\partial_\infty\HH^3\to \R$ be such that $c_{m_1}=e^{2y}c_{m_0}$. 
  Let $\xi\in \partial_\infty\HH^3$ be such that the angle at $m_0$ between the segment ending at $m_1$ and the geodesic ray ending at $\xi$ is $\theta$. Then
  \begin{equation}
    \label{eq:cosh}
     y(\xi)=-\log(\cosh(\delta)-\cos(\theta)\sinh(\delta))~. 
  \end{equation}
  As a consequence, if $\partial_\theta$ is the vector at $\xi$ corresponding to varying $\theta$ while keeping $m_0, m_1$ and the plane containing $m_0, m_1$ and $\xi$ fixed, then
    \begin{equation}
      \label{eq:der}
      dy(\partial_\theta)=-\frac{\sin(\theta)}{\rm{cotanh}(\delta)-\cos(\theta)}~. 
    \end{equation}
\end{proposition}

\begin{proof}
  The first point is a consequence of a standard hyperbolic geometry formula relating the values at $m_0$ and $m_1$ of the Busemann function associated to $\xi$. Namely, if $m(t)$ is the point at distance $t$ from $m_0$ along the ray $[m_0,\xi)$, and if $l(t)=d(m_1, m(t))$ then the triangle formula can be written as:
    $$  \cosh(l(t))=\cosh(t)\cosh(\delta)-\cos(\theta)\sinh(t)\sinh(\delta)~. $$
    As $t\to \infty$, it follows that
    $$ e^{l(t)-t} \sim \cosh(\delta)-\cos(\theta)\sinh(\delta)~, $$
  and the first point follows.

  The second point follows from the first by a direct computation.
\end{proof}

The Thurston metric has a simple relation to the visual metric, stated for the reader's convenience in the next proposition. We denote by $\Sigma_0$ the pleated locally convex surface facing $U_0$, defined as:
$$ \Sigma_0 = \partial CH(\CP^1\setminus U_0)~, $$
where $CH$ denotes the hyperbolic convex hull. 

\begin{proposition} \label{pr:Thm}
  Let $\xi\in U_0$, and let $m$ be the closest-point projection of $\xi$ on $\Sigma_0$. Then
  $$ h_{Th}(\xi)=c_m(\xi)~. $$
\end{proposition}

\begin{proof}
  Let $P$ be the support plane of $CH(\CP^1\setminus U_0)$ at $m$ orthogonal to the direction of the geodesic ray from $m$ to $\xi$. This plane $P$ faces a disk in $\CP^1$ (with the same boundary) which is a maximal disk contained in $U_0$ and containing $\xi$. By construction, $h_{Th}$ is equal at $\xi$ to the push-forward of the induced metric on $P$ by the normal exponential map, considered as a map from $P$ to $U_0$.

  However a simple computation (for instance done in the Poincaré model of $\HH^3$ with $m$ at the center) shows that this push-forward metric coincides at $\xi$ with the visual metric $c_m$, and the proposition follows.
\end{proof}

\begin{proof}[Proof of Lemma \ref{lm:x}]
    To obtain the required comparison, we choose a point $\xi\in U_0$ and let $m$ be the closest-point projection of $\xi$ on $\Sigma_0$, as above. We first compare $h_{Th}$ to $c_m$ at first-order at $\xi$, and then compare $c_m$ to $c_1$. Both comparisons will be based on \eqref{eq:cosh}.

    For the first comparison, we again denote by $u:U_0\to \R$ the function such that $h_{Th}=e^{2u}c_m$, and will show that $du=0$ at $\xi$. Consider first a first-order variation of $\xi$ obtained by varying $m$ along a flat direction of $\Sigma_0$ while keeping parallel the direction at $m$ of the closest-point projection line. Let $\partial_\delta$ be such a unit (for $c_m$) tangent vector to $\CP^1$ at $\xi$. Then taking the derivative of \eqref{eq:cosh} with respect to $\delta$ at $\delta=0$ and $\theta=\pi/2$ shows that $dy(\partial_\delta)=0$. Now if $m$ is located on a pleating line of $\Sigma_0$ and $\partial_\theta$ is a vector corresponding to varying $\theta$ while the projection to $m$ remains constant, then taking the derivative of \eqref{eq:cosh} with respect to $\theta$ with $\delta=0$ and $\theta=\pi/2$ also shows that $dy(\partial_\theta)=0$. So we can conclude that $dy=0$ at $\xi$, so that $h_{Th}$ coincides with $c_m$ at $\xi$ at first order.

  The second step is to compare at $\xi$ the metric $c_m$ -- which as we have seen is equal to $h_{Th}$ at first order -- to the spherical metric $c_1$, which as any spherical metric on $\partial_\infty \HH^3$ is the visual metric of some fixed point $m_1\in \HH^3$. Specifically, if we denote (again) by $y:\partial_\infty\HH^3\to \R$ the function such that $c_m=e^{2y}c_{m_1}$, Proposition \ref{pr:m0m1} shows that the derivative of $y$ at $\xi$ is bounded by the right-hand side of \eqref{eq:der}, with $m_0=m$. However a look at this second term shows that it is bounded from above in absolute value as soon as $\theta$ is bounded from below, because $\rm{cotanh}(\delta)-\cos(\theta)$ is then bounded from below since $\rm{cotanh}(\delta)>1$. 

  But since we have assumed that $m_1\in CH(\CP^1\setminus U_0)$, for all $\xi\in U_0$, if $m$ again denotes the closest-point projection of $\xi$ on $\Sigma_0$, $m_1$ is on the negative side of the support plane to $CH(\CP^1\setminus U_0)$ at $m$ orthogonal to the geodesic ray from $m$ to $\xi$. (That is, $m_1$ is on the side of this plane not containing $\xi$.) It follows that the angle $\theta$ occuring in \eqref{eq:der} is larger than $\pi/2$. This concludes the proof.
\end{proof}

\subsection{Construction of approximating metrics}
\label{ssc:construction}

We can now define rather explicitly the sequence of approximating metrics that will be used in the proof. The construction uses the choice of a ``cut-off'' function.

\begin{lemma} \label{lm:phi}
  There exists a smooth function $\phi:\R\to \R$ such that:
  \begin{enumerate}
  \item $\phi(x)=x$ for $x\leq 0$,
  \item $\phi(x)=1$ for $x\geq 2$,
  \item $\phi'(x)\in [0,1]$ and $\phi''(x)\leq 0$ for all $x\in [0,2]$,
  \item $\phi'(x) \leq e^{2(\phi(x)-x)}$ for all $x\in [0,2]$.
  \end{enumerate}
\end{lemma}

\begin{proof}
  Note first that $f(x)=x$ is a solution of the differential equation
  $$ f'(x) \leq e^{2(f(x)-x)} $$
  over $[0,1]$. As a consequence, one can find a function $\phi$ satisfying the conditions above over $(-\infty, 1-\epsilon)$ and arbitrarily close to $f(x)=x$, for any $\epsilon>0$. Taking $\epsilon$ small enough, one can then ``bridge'' this function to $f(x)=1$ for $x\geq 1+\epsilon$ while still satisfying the conditions of the lemma.
\end{proof}

\begin{definition}
  We choose a function $\phi$ satisfying the conditions of Lemma \ref{lm:phi}.
  For all $n\in \N$, we define $\phi_n:\R\to \R$ by 
  $$ \phi_n(x) = n+\phi(x-n)~. $$
\end{definition}

For the next definition, we need to remark that $\CP^1\setminus U$ has at least 2 points, because $U$ admits a complete, conformal metric of negative curvature. Therefore the convex hull of $\CP^1\setminus U$ contains a point $m_1$. 

\begin{definition} \label{df:hn}
  Let $(U,h)$ be a boundary data on $\CP^1$, let $c_1=c_{m_1}$ be the visual metric of a point $m_1\in CH(\CP^1\setminus U)$, and let $u:U\to \R$ be such that $h=e^{2u}c_1$. We define a sequence of conformal metrics $(h_n)_{n\in \N}$ on $\CP^1$ by
  \begin{eqnarray*}
    h_n=e^{2\phi_n(u)} c_1 & & \mbox{on $U$} \\
    h_n=e^{2(n+1)}c_1 & & \mbox{on $\CP^1\setminus U$}~.
  \end{eqnarray*}
\end{definition}

Note that $h_n$ is a smooth metric on $\CP^1$ since, by definition of $\phi$, $\phi_n(u)$ is equal to $n+1$ in a neighborhood of the boundary of $U$. (It follows from the fact that $h$ is complete on $U$ together with Lemma \ref{lm:conformal-bounds} (for from Lemma \ref{lm:v}) that $u\to \infty$ on $\partial U$.)

\begin{lemma} \label{lm:Kmax}
  Assume that $K\in [-1+\epsilon,-\epsilon]$ with derivatives uniformly bounded (with respect to $h$), that $u$ is bounded from below, and that $\| du\|_h$ is bounded from above, on $U$.
  There exists a constant $K_{max}>0$ (depending on those bounds) such that the curvature $K_n$ of $h_n$ is:
  \begin{itemize}
  \item equal to $K$ on the set of points $x\in U$ where $u(x)\leq n$,
  \item in the interval $[-1+\epsilon, K_{max}]$ on the set of points $x\in U$ where $n\leq u(x)\leq n+2$, 
  \item equal to $e^{-2(n+1)}$ on the set of points  $x\in U$ where $u(x)\geq n+2$ and on the complement of $U$. 
  \end{itemize}
  Moreover, the metrics $h_n$ have bounded derivatives, for constants $C_k$ which do not depend on $n$.
\end{lemma}

\begin{proof}
  In the region of $U$ where $u(x)\leq n$, $h_n=h$, and therefore $K_n=K$, the curvature of $h$. On the region of $U$ where $u(x)\geq n+2$, and on the complement of $U$, $h_n=e^{2(n+1)}c_1$, and it follows from Lemma \ref{lm:conformal} that $K_n=e^{-2(n+1)}$.
  
  We now focus on the region where $n\leq u(x)\leq n+2$.
  The conformal transformation of curvature (Lemma \ref{lm:conformal}) also shows that
  $$ K_n = e^{-2\phi_n(u)}(\Delta_{c_1}\phi_n(u)) + 1)~. $$
  However
  $$ d(\phi_n(u)) = \phi'_n(u) du~, $$
  so that
  $$ Dd(\phi_n\circ u) = \phi'_n(u) Ddu + (\phi''_n\circ u) du\otimes du~, $$
  and therefore
  $$ \Delta_{c_1}(\phi_n\circ u) = (\phi'_n\circ u) \Delta_{c_1}u - (\phi''_n\circ u) \| du\|^2_{c_1}~. $$
  As a consequence, and since $\| du\|^2_{c_1}=e^{2u}\| du\|^2_h$:
  \begin{eqnarray*}
    K_n & = & e^{-2\phi_n\circ u} ((\phi'_n\circ u) \Delta_{c_1}u - (\phi''_n\circ u) \| du\|^2_{c_1} +1) \\
        & = & (\phi'_n\circ u) e^{2(u-\phi_n\circ u)} e^{-2u} (\Delta_{c_1}u+1) + e^{-2\phi_n\circ u}(1-(\phi'_n\circ u)) - e^{2(u-\phi_n\circ u)} (\phi''_n\circ u) \| du\|^2_h \\
    & = & (\phi'_n\circ u) e^{2(u-\phi_n\circ u)} K + e^{-2\phi_n\circ u}(1-(\phi'_n\circ u)) - e^{2(u-\phi_n\circ u)} (\phi''_n\circ u) \| du\|^2_h~.
  \end{eqnarray*}
  By the hypothesis, $K\in [-1+\epsilon,-\epsilon]$, and by point (4) in the choice of $\phi$, $0\leq (\phi'_n\circ u) e^{2(u-\phi_n\circ u)}\leq1$, so the first term is in $[-1+\epsilon,0]$. By point (3) in the choice of $\phi$, the second term is non-negative, and it is uniformly bounded because by construction, $e^{-2\phi_n\circ u}$ is bounded from above. Finally, the third term is also non-negative by point (3) in the choice of $\phi$ and uniformly bounded by the fact that $u-\phi_n\circ u$ is bounded from above and the uniform bound on $\| du\|_h$. The first part of the result follows.

  For the second part, note that the metric $c_1$ can be written as $c_1=e^{-2u}h$. It thus follows from Lemma \ref{lm:conformal} that
  $$ 1 = e^{2u}(-\Delta_hu+K)~, $$
  so that
  $$ \Delta_hu = e^{-2u}+K~. $$
  Elliptic regularity and a standard bootstrapping argument, together with the hypothesis that $\| du\|_h$ is uniformly bounded, then shows that all derivatives of $u$ are uniformly bounded with respect to $h$.

  We can then deduce from those estimates on $u$ that similar estimates also apply to $\phi_n(u)$, that is, its derivatives are uniformly bounded with respect to $h_n$. Indeed, this is clear in the region where $u\leq n$, and also in the region where $u\geq n+2$. In the region where $n\leq u\leq n+2$, it follows from the definition of $\phi_n$ (and from the uniform bound on $\| du_n\|$) that the derivatives of $\phi_n(u)$ are uniformly bounded with respect to $h$. Since $e^{-4}h\leq h_n\leq h$ when $n\leq u\leq n+1$, the derivatives of $\phi_n(u)$ are also uniformly bounded with respect to $h_n$. Lemma \ref{lm:conformal} then shows that the derivatives of $K_n$ are uniformly bounded with respect to $h_n$. 
\end{proof}

\subsection{A lower bound on the injectivity radius}
\label{ssc:bounds}

In addition to the upper bound on the curvature, the next section will make an essential use of a lower bound on the injectivity radius of the metrics $h_n$, as stated in the next lemma.

\begin{lemma} \label{lm:inj}
  There exists $l_0>0$ such that the injectivity radius of $h_n$ is bounded from below by $l_0$.  
\end{lemma}

We introduce additional notations for the proof, namely a decomposition of $\CP^1$ in three subsets:
\begin{itemize}
\item $U_{u\leq n}$ is the set of points $x\in U$ where $u(x)\leq n$,
\item $U_{n<u\leq n+1}$ is the set of points $x\in U$ where $n<u(x)\leq n+1$,
\item $U_{u\leq n+1}=U_{u\leq n}\cup U_{n<u\leq n+1}$,
\item $U_{n+1<u}$ is the union of $\CP^1\setminus U$ with the set of points $x\in U$ where $n+1<u(x)$.
\end{itemize}

\begin{proof}[Proof of Lemma \ref{lm:inj}]
  We consider a short geodesic $\gamma$ of length $l$ for the metric $h_n$, and will show that $l$ cannot be smaller than $l_0$, for a value of $l_0$ that will be determined. 

  The proof will consider two cases, with a similar argument used for a different background metric in each case. In each case we will use that any disk $\Omega$ with $\partial\Omega=\gamma$ must have total curvature $\pi$ (by the Gauss-Bonnet Theorem) and thus, since $h_n$ has curvature $K\leq K_{max}$, area bounded from below: $A_{h_n}(\Omega)\geq \pi/K_{max}$.

  \medskip

  {\em First case: $\gamma\cap U_{u\leq n}=\emptyset$.}
  In this case $u\geq n$ on $\gamma$, and we compare $h_n$ to the metric $e^{2n}c_1$. We have $h_n\geq e^{2n}c_1$ on $\gamma$, so that $L_{c_1}(\gamma)\leq e^{-n}l$. However we also have $h_n\leq e^{2(n+1)}c_1$ on $\CP^1$ and therefore, for each of the disks $\Omega$ bounded by $\gamma$, $A_{c_1}(\Omega)\geq e^{-2(n+1)}\pi/K_{max}$. For $l$ small enough, this contradicts the spherical isoperimetric inegality.
  
  \medskip

  {\em Second case: $\gamma\cap U_{u\leq n}\neq\emptyset$.}
  We first notice that there exists a constant $l_1>0$ (depending on the constant $C$ in Lemma \ref{lm:conformal-bounds}) such that, if $L_{h_n}(\gamma)\leq l_1$, then $L_h(\gamma)\leq 2l_1$ and $u\leq n+2$ on $\gamma$. Indeed, $\| du\|_h\leq C$, and since
  $$ h=e^{2u}c_1 = e^{2u-2\phi_n(u)}h_n \leq e^{2u-2n}h_n~, $$
  we have
  $$ \| du\|_{h_n} \leq e^{u-n}\| du\|_h \leq e^{u-n} C~. $$
  It follows from this equation, and from the fact that $u\leq n$ at a point of $\gamma$, that if $\gamma$ is short enough, then $e^{u-n}\leq 2$ on $\gamma$, and therefore $L_h(\gamma)\leq 2L_{h_n}(\gamma)\leq 2l$.  

  It then follows that $\gamma\subset U$, and therefore $\gamma$ is contained in a connected component $U_0$ of $U$. Since $U$ is a disjoint union of quasidisks, this connected component is topologically a disk, and therefore $\gamma$ bounds a topological disk $\Omega\subset U_0$. Since $h\geq h_n$ on $U_0$, $A_h(\Omega)\geq A_{h_n}(\Omega)\geq \pi/K_{max}$. So for $l$ small enough, we again have a contradiction with the isoperimetric inequality, this time for $(U_0,h)$.
\end{proof}

\subsection{A sequence of convex subsets}
\label{ssc:subsets}

We can now apply the Alexandrov Theorem to the sequence of metrics constructed in Section \ref{ssc:construction}, together with the injectivity radius bounds in Section \ref{ssc:bounds}, to obtain the following lemma.

\begin{lemma} \label{lm:sequence}
  Under the hypothesis of Theorem \ref{tm:main}, there exists a sequence of bounded convex domains $\Omega_n\subset \HH^3$, together with a sequence of diffeomorphisms $\phi_n:\CP^1\to \partial\Omega_n$, such that for all $n\in \N$:
  \begin{itemize}
  \item$\phi_n$ is a conformal map from $\CP^1$ to $\partial\Omega_n$, 
  \item $\phi_n^*I=h$ on $U_n$, where $U_n\subset U$, $U_n\subset U_m$ for $n\leq m$, and $\cup_{n\in \N}U_n=U$,
  \item $\phi_n^*I$ has constant curvature $e^{-2n}$ on $\CP^1\setminus U$, 
  \item the induced metric on $\partial\Omega_n$ has curvature in $[-1+\epsilon,K_{max}]$, 
  \item the injectivity radius of $(\partial \Omega_n, I)$ is bounded from below by a constant $l_0$,
  \item the induced metric on $\partial\Omega_n$ has bounded derivatives (in the sense of Definition \ref{df:bounded}).
  \end{itemize}
\end{lemma}

 %% Approximation by metrics on the sphere

    \section{Bounds on principal curvatures}
\label{sc:bounds}

\subsection{Principal curvature bounds for convex surfaces}
\label{ssc:principal}

The first point of this section is to find a uniform bound on the principal curvatures on the boundaries of the domains $\Omega_n$ constructed in Lemma \ref{lm:sequence}. 

\begin{lemma} \label{lm:p-curvature}
  Let $\epsilon>0$. There exists a constant $k_0>0$ such that if $\Omega\subset \HH^3$ is a convex subset such that the induced metric on $\partial\Omega$ has curvature $K\in [-1+\epsilon, 1/\epsilon]$, bounded derivatives, and injectivity radius bounded from below by $\epsilon$, then the principal curvatures of $\partial\Omega$ are in $[k_0,1/k_0]$.
\end{lemma}

The proof uses a result of Labourie \cite[Théorème D]{L1} which we recall here for the reader's convenience (restricting to the special case where the target space is $\HH^3$).

\begin{theorem}[Labourie] \label{tm:lab-compact}
  Let $f_n:S\to \HH^3$ be a sequence of immersions of a surface $S$ such that the pullback $f_n^*(h)$ of the hyperbolic metric $h$ converges smoothly to a metric $g_0$.
If the integral of the mean curvature is uniformly bounded, then a subsequence of $f_n$ converges smoothly to an isometric immersion
$f_\infty$ such that  $f_\infty^*(h)=g_0$.
\end{theorem}

Note that Theorem \ref{tm:lab-compact} is local: no assumption of compactness of $S$ or completeness of $f_n^*(h)$ is needed.

\begin{proof}[Proof of Lemma \ref{lm:p-curvature}]
  We proceed by contradiction and assume that there exists $\epsilon>0$ such that for all $n$, there exists a convex subset $\Omega_n$ such that the induced metric on $\partial\Omega_n$ has curvature $K\in [-1+\epsilon,1/\epsilon]$ and injectivity radius bounded from below by $\epsilon$, while the principal curvatures of $\partial\Omega_n$ at a point $x_n$ are not in $[1/n,n]$.
  
  We apply Theorem \ref{tm:lab-compact} to a sequence of parameterizations $\rho_n:\DD\to \partial \Omega_n$ of the disk of center $x_n$ and radius $\epsilon/2$, which are indeed topological disks by the asumption on the injectivity radius. Since the induced metrics on the $\rho_n(\DD)$ have curvature $K\in [-1+\epsilon, 1/\epsilon]$, there exists a subsequence such that the induced metrics converge to a limit $g_0$. Since the $\Omega_n$ have bounded curvature, we can choose the $\rho_n$ so that this convergence is $C^\infty$. Therefore, by Theorem \ref{tm:lab-compact}, after extracting another subsequence, there exists a sequence of isometries $(\psi_n)_{n\in \N}$ of $\HH^3$ such that
  \begin{enumerate}
  \item either the integral mean curvature on $\rho_n(\DD)$ tends to infinity,
  \item or the $\psi_n\circ\rho_n:\DD\to \HH^3$ converge smoothly to a smooth limit $\rho_\infty:\DD\to \HH^3$. 
  \end{enumerate}
  Clearly the second hypothesis is excluded by the hypothesis on the principal curvatures of $\partial\Omega_n$ at $x_n$, so it is sufficient to exclude also the first case.

  Let $\Omega_n^1$ be the $1$-neighborhood of $\Omega_n$:
  $$ \Omega_n^1 = \{ x\in \HH^3 ~|~ d(x,\Omega_n)\leq 1 \}~, $$ 
  and let $\rho:\partial\Omega_n^1\to \partial\Omega_n$ be the closest-point projection. Since $\Omega_n$ is convex and smooth, $\Omega_n^1$ is also convex, and $\rho$ is a diffeomorphism. A well-known ``tube'' formula indicates that $\rho$ contracts the area of $\partial\Omega_n^1$ at a point $x\in \partial\Omega_n^1$ by a factor $\cosh(1)^2 + \cosh(1)\sinh(1)\tr(B) + \sinh(1)^2\det(B)$, where $B$ denotes the shape operator of $\Omega_n$ at $\rho(x)$. As a consequence, the area of $\rho^{-1}(\rho_n(\DD))$ is equal to
  $$ A(\rho^{-1}(\rho_n(\DD))) = \int_{\rho_n(\DD)} \cosh(t)^2 + \cosh(t)\sinh(t)\tr(B) + \sinh(t)^2\det(B) da_I~. $$
  Let $B(x_n,2)$ be the ball of center $x_n$ and radius $2$. The projection from
  $$ (\partial B(x_n,2))\cap (\HH^3\setminus \Omega^1_n) $$
  to $\partial \Omega^1_n$ is contracting, and its image contains $\rho_n(\DD)$. As a consequence, the area of $\rho^{-1}\circ\rho_n(\DD)$ is at most equal to the area of $\partial B(x_n,2)$, so that
  $$ \int_{\rho_n(\DD)}\cosh(1)^2 + \cosh(1)\sinh(1)\tr(B) + \sinh(1)^2\det(B) da_I \leq 4\pi \sinh(2)^2~. $$
  Since all three terms are non-negative, this simplies in particular that the integral of the mean curvature is bounded on $\rho_n(\DD)$, and the first alternative above is impossible. This concludes the proof of the lemma.
\end{proof}

\subsection{Curvature bounds in the almost-flat regions}

In the proof of Theorem \ref{tm:main} it will be necessary to show that the conformal structure on the boundary at infinity of the convex set $\Omega_\infty$ which is the limit of the $\Omega_n$ is identified -- under the limit map from $\CP^1$ to $\overline{\HH^3}$ -- with the conformal structure at infinity of $\CP^1$. This will be based on the fact that the regions in the $\partial\Omega_n$ which are the image of the complement of $U$ are (asymptotically) umbilic, with principal curvatures going to $1$ as $n\to\infty$. This section is focused on proving this result.

Specifically, we will need the following lemma.

\begin{lemma} \label{lm:flat-regions}
  For all $k_0>0$, there exists $R>0$ and $K_0>0$ such that if $\Omega\subset \HH^3$ is a convex domain and $x\in \partial \Omega$ such that the injectivity radius of $(\partial \Omega,I)$ at $x$ is at least $R$, with curvature in $[-K_0,K_0]$ on the ball of center $x$ and radius $R$, then the principal curvatures of $\partial\Omega$ at $x$ are in $[1-k_0,1+k_0]$.
\end{lemma}

The proof uses the following proposition.

\begin{proposition} \label{pr:volkov}
Let $\Omega\subset \HH^3$. Assume that $(\partial \Omega, I)$ is flat, and that its injectivity radius at a point $x\in  \partial \Omega$ is infinite. Then $\Omega$ is a horosphere. 
\end{proposition}

\begin{proof}
  This follows directly from a result of Volkov and Vladimirova \cite{volkov-vladimirova} who proved that an isometric immersion of the Euclidean plane in $\HH^3$ is either a parameterization of a horosphere, or a multiple cover of a cylinder. 
\end{proof}

\begin{proof}[Proof of Lemma \ref{lm:flat-regions}]
  Assume by contradiction that there exists $k_0>0$ such that no such $R$ and $K_0$ exist. Then for all $n\in \N_{>0}$ there exists a convex subset $\Omega_n$ and a point $x_n\in \partial \Omega_n$ such that the injectivity radius of $\partial \Omega_n$ at $x_n$ is at least $n$, and the curvature of $\partial\Omega_n$ on the ball of center $x_n$ and radius $n$ is in $[-1/n, 1/n]$. However one of the principal curvatures of $\partial\Omega_n$ at $x_n$ is not in $[1-k_0,1+k_0]$.

  For each $n$, let $\phi_n\in \Isom(\HH^3)$ be an isometry sending $x_n$ to $x_0$ (a fixed point). We can extract from $(\phi_n(\Omega_n))_{n\in \N}$ a sequence of convex subset $\Omega'_n\subset \HH^3$ which converges in the Hausdorff topology on compact subsets of $\HH^3$ to a limit $\Omega_\infty$ (see \cite{nadler-quinn-stavrakas} and references therein for topological properties of the space of convex subsets of a linear space).

  Then $x_0\in \partial\Omega_\infty$ by construction. Moreover $\partial\Omega'_n\to \partial\Omega_\infty$, and the induced metric on $\partial\Omega'_n$ converges to the induced metric on $\partial\Omega_\infty$ (see \cite[Section 10.2]{burago-burago-ivanov}). As a consequence, $\partial\Omega_\infty$ is intrisically flat, with infinite injectivity radius since the injectivity radius of $\partial\Omega_n$ at $x_n$ is at least $n$. It follows from Proposition \ref{pr:volkov} that it is a horosphere. This contradicts the hypothesis that the principal curvatures of $\partial\Omega'_n$ are not in  $[1-k_0,1+k_0]$ at $x_0$.
\end{proof}

%\subsection{Convex domains for the dual theorem} %% Bounds on the principal curvatures

    \section{Proofs of the main results}
\label{sc:proofs}

We are now ready to prove Theorem \ref{tm:main}.

\subsection{Convergence of the Gauss maps}

We consider the sequence of convex domains $\Omega_n$ and conformal diffeomorphisms $\phi_n:\CP^1\to \partial\Omega_n$ described in Lemma \ref{lm:sequence}.

For each $n$, we denote by $G_n:\partial \Omega_n\to \partial_\infty \HH^3$ the hyperbolic Gauss map of $\partial\Omega_n$, that is, the map sending $x\in \partial\Omega_n$ to the endpoint at infinity of the geodesic ray starting from $x$ directed by the exterior normal of $\Omega$ at $x$.

\begin{lemma} \label{lm:gauss}
  There exists a constant $K_G>1$ such that for all $n\in \N$, the map $G_n\circ \phi_n:\CP^1\to \partial_\infty\HH^3$ is $K_G$-quasi-conformal. 
\end{lemma}

\begin{proof}
  Let $x\in \partial\Omega_n$, for some $n\in \N$. We can identify $T_x\partial \Omega_n$ to $T_{G_n(x)}\partial_\infty\HH^3$ by the exponential map at $x$, $\exp_x:T_x\HH^3\to \partial_\infty\HH^3$, using the natural identification of $T_x\partial\Omega_n$ with the normal in $T_x\HH^3$ of the unit normal to $\partial\Omega_n$ at $x$. Under this identification, the differential of $G_n$ at $x$ is equal to $Id+B$, where $Id$ is the identity, and $B$ the shape operator of $\partial\Omega_n$.

  Since the principal curvatures of $\partial\Omega_n$ are uniformly bounded, the differentials of the Gauss map $G_n$ are uniformly bounded. Since the $\phi_n$ are conformal, it follows that the maps $G_n\circ\phi_n$ are uniformly quasi-conformal. 
\end{proof}

\begin{corollary} \label{cr:compact}
  After extracting a subsequence and composing with a sequence of M\"obius transformations $u_n$, the sequence $(G_n\circ \phi_n)_n$  converges $C^0$ to a $K_G$-quasi-conformal homeomorphism $\psi_\infty:\CP^1\to \partial_\infty \HH^3$. 
\end{corollary}

\begin{proof}
  This follows from the compactness of the space of $K$-quasi-conformal homeomorphisms from $\CP^1$ to $\CP^1$, see e.g. \cite{ahlfors}.
\end{proof}

Note that the M\"obius transformations appearing in this corollary are boundary values of hyperbolic isometries. If $u\in PSL(2,\C)$ is a hyperbolic isometry, and $\partial u$ is its boundary action, considered as a M\"obius transformation, then composing $G_n\circ \phi_n$ on the left by a M\"obius transformation $\partial u$ is equivalent to composing $\phi_n$ on the left by $u$ and then composing with its Gauss map.

We will assume in the last section below that the $\phi_n$ have already been composed on the left by a sequence of M\"obius transformations, so that $G_n\circ \phi_n\to \psi_\infty$. 

\subsection{Convergence of boundary data}

Recall that we denote by $\rho: \HH^3\to B(0,1)$ the Poincar\'e model of hyperbolic space. We consider the sequence of maps $(\phi_n)_{n\in \N}$ considered above.

\begin{lemma} \label{lm:final}
  After extracting a subsequence,
  \begin{enumerate}
  \item the sequence of restrictions of the $\phi_n$ to $U$ converge $C^\infty$ to an isometric embedding $\phi_\infty$ of $(U,h)$ in $\HH^3$,
  \item the sequence $(\rho\circ \phi_n)_{n\in \N}$ converges to a map $\Phi_\infty:\CP^1\to B(0,1)$ which is conformal on $U$ and on the interior of $\CP^1\setminus U$,
  \item $\phi_\infty(U)$ is the boundary in $\HH^3$ of a convex subset $\Omega$.
  \end{enumerate}
\end{lemma}

\begin{proof}
  Let $U_0$ be a connected component of $U$. It follows from the arguments already used in Section \ref{ssc:principal} that there exists a sequence $(u^0_n)_{n\in \N}$ of hyperbolic isometries such that the sequence of restrictions of the $u^0_n\circ \phi_n$ to $U_0$ converges $C^\infty$ (after extracting a subsequence) to a limit $\phi_{0,\infty}$. It follows that the Gauss maps of the $u^0_n\circ \phi_n$ converge to the Gauss map of $\phi_{0,\infty}$ on $U_0$, that is,
  $$ \partial u_n^0 \circ(G_n\circ \phi_n) \to G_{0,\infty}\circ \phi_{0,\infty}~, $$
  where $G_{0,\infty}$ denotes the Gauss map of $\phi_{0,\infty}$. 

  However we already know by Corollary \ref{cr:compact} that the $G_n\circ \phi_n$ converge to a limit $\psi_\infty$, so that
  $$ \lim_{n\to \infty} \partial u_n^0\circ \psi_\infty=G_{0,\infty}\circ \phi_{0,\infty}~. $$
  As a consequence, $(u_n^0)_{n\in \N}$ converges to a hyperbolic isometry $u^0$. So, to obtain the convergence of $u_n^0\circ \phi_n\circ \phi_{0,\infty}$ above, it is sufficient to take $u_n^0$ constant and equal to $u^0$. This means that after replacing the $\phi_n$ by $u^0\circ \phi_n$, the sequence $(\phi_n)_{n\in \N}$ converges $C^\infty$ (without normalization by the sequence $(u^0_n)_{n\in \N}$). This proves the first point. 

  The second point is less obvious than it might appear. We already know that on $U$, $\phi_n\to \phi_\infty$, where $\phi_\infty$ is an isometry from $(U,h)$ to $\HH^3$ and is therefore conformal. However the conformality of the limit map is less clear on $\CP^1\setminus U$. This is because the Poincar\'e model of $\HH^3$ does not preserve convexity, so that the $\rho(\Omega)$ are not necessarily convex, and a sequence of conformal embeddings might well converge to a limit which is not conformal. This could conceivably happen at parts of $\partial\Omega_n$ converging to infinity in $\HH^3$.

  To avoid this issue for regions converging to $\partial_\infty \HH^3$, we note that for all $x\in \CP^1\setminus U$, $(\phi_n(x))_{n\in \N}$ has the same limit in $\partial_\infty \HH^3$ as $(G_n\circ \phi_n(x))_{n\in \N}$. Let $x\in \CP^1$ be contained in the interior of the complement of $U$, then Lemma \ref{lm:flat-regions} shows that the principal curvatures of $\phi_n$ at $x$ converge to $1$ as $n\to \infty$. It follows that the pull-back by $G_n\circ\phi_n$ of the conformal metric at infinity is quasi-conformal to $h_n$ in the neighborhood of $x$, with quasiconformal constant going to $1$ as $n\to \infty$. So $\phi_\infty$ is conformal on $U$ and on the interior of $\CP^1\setminus U$. 

  The third point follows from the construction, since $\phi_\infty(U)=\partial\Omega_\infty$, where $\Omega_\infty$ is the Hausdorff limit (on compact subsets of $\HH^3$) of the $\Omega_n$, which are convex subsets of $\HH^3$. 
\end{proof}

\begin{proof}[Proof of Theorem \ref{tm:main}]
  Let $(U,h)$ be a boundary data, with $U\subset \CP^1$ a disjoint union of $k_0$-quasidisks, and let $h$ be a complete metric of curvature $K\in [-1+\epsilon,-\epsilon]$ on $U$. We define a sequence $(h_n)_{n\in \N}$ of approximating metrics following Definition \ref{df:hn}, and those metrics have curvature $K_n\in [-1+\epsilon,K_{max}]$ by Lemma \ref{lm:Kmax}. (Note that the constant $K_{max}$ does not depend on the connected component of $U$ since each is a $k_0$-quasidisk, and the curvature estimates on $h_n$ are valid for all.)

  The Alexandrov Theorem (Theorem \ref{tm:alexH}) therefore realizes each of the $h_n$ as the induced metric on the boundary of a convex subset $\Omega_n\subset\HH^3$, as see in Section \ref{ssc:subsets}.

  Lemma \ref{lm:final} then shows that, after normalizing by a sequence of isometries and extracting a subsequence, $(\Omega_n)_{n\in \N}$ converges to a convex subset $\Omega_\infty$ realizing the boundary data $(U,h)$. 
\end{proof} %% proofs of the main results

\bibliographystyle{halpha}
\bibliography{/home/jean-marc/Dropbox/papiers/outils/biblio}
\end{document}